\documentclass[a4paper,11pt,twoside]{article}

\usepackage[body={175mm,250mm}]{geometry}
\usepackage{a1}
\usepackage[english]{babel}
\usepackage[latin1]{inputenc}
\usepackage{amsfonts,amssymb}
\usepackage{amsmath}
\usepackage{graphicx}	

\usepackage{makeidx}
\makeindex

\def\mapright#1#2#3{\smash{\mathop{\hbox to
#3{\rightarrowfill}}\limits^{#1}_{#2}}}
\def\mapleft#1#2#3{\smash{\mathop{\hbox to
#3{\leftarrowfill}}\limits^{#1}_{#2}}}
\def\mapright#1#2{\smash{\mathop{\hbox to
0.90cm{\rightarrowfill}}\limits^{#1}_{#2}}}
\def\mapleft#1#2{\smash{\mathop{\hbox to
0.90cm{\leftarrowfill}}\limits^{#1}_{#2}}}
\def\mapleftright#1#2{\smash{\mathop{\hbox to
0.80cm{\leftarrowfill \rightarrowfill}}\limits^{#1}_{#2}}}

\def\paf{\partial f}
\def\paS{\partial S}
\def\C{\mathbb C}

\def\Z{\mathbb Z}

\def\cP{\mathcal{P}}
\def\cR{\mathcal R}
\def\Rsq{\cR^{\sqrt{}}}

\title{A state-sum invariant of tangles in oriented
surfaces}
\author{Peter M.~Johnson \and S\'ostenes Lins}
\date{}

\begin{document}

\maketitle

\begin{abstract}
We exploit a recent insight of \cite{joli2012A}, which explains how to
obtain invariants of links in space, or tangles in surfaces, from
quantities invariant only under a restricted set of Reidemeister moves.
The main idea involves modifying diagrams to simplify their faces.
This will now be used to define new state-sum invariants based on
assigning symbols to faces, in a new way that avoids 
undesirable simplifications in the relations.
Any such invariant has several properties of interest,
among them functorial ones roughly like those of a TQFT:
pasting surfaces with compatible tangles along some of their boundary
components corresponds to obtaining the value of a matrix-valued invariant
by multiplying matrices for the individual pieces. 

As a concrete illustration of the ideas, simple assumptions yield what we
call the $u$-invariant.  For closed tangles (links in surfaces) it
takes values in $\Z[u]$, where $u$ is a primitive fifth root of unity.
It is well-adapted for computations, surprisingly strong for something
whose definition (with verification) is so easy, and has several
interesting properties.
\end{abstract}

\section{Introduction}

While this work is mainly concerned with invariants of links, it
has its inspiration in the 3-manifold invariants introduced in
two seminal articles: that by Turaev and Viro \cite{turaev1992state},
where symbols in $\{0,\ldots,r-1\}$ are assigned to the 2-cells
of a special spine of a 3-manifold, and that by Reshetikhin and Turaev 
\cite{reshetikhin1991invariants}, where such symbols are assigned to
the faces and component knots of a projection to $S^2$ of a framed link
in $S^3$ representing a 3-manifold by surgery instructions.
A state is such an assignment, subject to certain rules.
See also Lickorish \cite{lickorish1991three} and Kauffman-Lins
\cite{kauffman1994tlr} for connections of the Reshetikhin-Turaev
theory with the Temperley-Lieb algebra and the Jones polynomial.
Our motivating idea was to generalize and simplify the above constructions 
via an abstract algebraic approach, in a way roughly similar to the
treatment of ideal Turaev-Viro invariants by King \cite{king2007idealTV},
but with a more radical reduction of the machinery involved.
Euler characteristics of faces were relevant in \cite{turaev1992state}
and will reappear below in a related role.

What we now wish to present arose from that project.  It concerns a way to
produce invariants of tangles in oriented surfaces $S$, where the value of
each state of a link or tangle diagram $T$ is the product of variables that
record information about the state, at least on faces and around vertices. 
The value of a diagram is obtained by summing the values of all its states.
As usual, the notion of diagrams equivalent under moves leads to
state-sum invariants, in rings satisfying
relations obtained from pairs of diagrams.

To handle a problem related to a non-local property of Reidemeister moves of
type 2, we were led to re-examine the purely diagrammatic foundations of the
theory of links and tangles.  The conclusion, justified in our companion
article \cite{joli2012A}, is that to produce invariants it is enough to require
invariance under a restricted set of Reidemeister moves,
provided the tangle diagrams are first adjusted so as to be fine,
as defined in the next section.

The present article applies the main result of \cite{joli2012A} in order to
construct families of link and tangle invariants with a remarkable property,
much like one for topological quantum field theories:
calculations made locally, using tangles
in the pieces of any partition of the surface, can be merged (by multiplying
matrices), to produce the value of the invariant, without any reliance on 
global information.  Such computations are highly parallelizable.
In contrast, the Kauffman bracket seems to be inherently global, as its
calculation relies on obtaining the number of circles produced by each state.
However, some even more general invariants, such as those of Khovanov-Rozansky
homology, are amenable to local calculations, albeit difficult ones.
They use a canopolis formalism based on planar algebras,
so only plane projections of links are contemplated.
For details one can consult Webster \cite{webster2007} and the references therein.

A complete analysis of the invariants we consider was possible after an
especially simple abstraction from the case where the number $r$ of symbols
is 2.  Only two different non-trivial invariants arose.
One is known: it is essentially the state-sum invariant studied by Kauffman
in Part~II, Section~5 of \cite{kauffman1991knots}.
It computes nothing more than the absolute values of linking numbers of
unoriented links (in surfaces, if desired), and will not be discussed further.
The other, called the $u$-invariant, is new.  For links, it takes values
in $\Z[u]$, where $u$ is a primitive fifth root of unity.
Its definition is so easy that invariance can be checked by hand. 
We thought of the $u$-invariant as a toy test case until, after generating
tables, we saw that its power to discriminate is not much weaker than that
of the Kauffman bracket (Jones polynomial). 
Although extensive tests did not reveal any pair of links distinguished
by the $u$-invariant but not by the Jones polynomial, we were unable to
find any relation between values given by the two invariants.
The $u$-invariant is the more easily calculable one,
not just due to the locality property mentioned above, but because
its more restrictive rules mean that the number of admissible states
of link diagrams with $n$ crossings
($n+2$ faces if planar) tend to be only a small fraction of $2^n$.
A small table for knots, and related graphics, appear in the appendixes.
The last section develops some theory for the general invariants.

We are indebted to the Departamento de Matem\'atica, UFPE, Brazil and to the
Centro de Inform\'atica, UFPE, Brazil for financial support.
The second author is also supported by a research grant from CNPq/Brazil,
proc.~301233/2009-8.

\section{The basic topological objects}

Objects and maps are assumed throughout to be piecewise linear.
Let $S$ be a compact oriented surface of genus $g$ whose boundary $\paS$
has $c$ components, with orientation induced from that of $S$.
The Euler characteristic $\chi_S$ of $S$, usually defined from counts
made after cutting $S$ into simple pieces, is $2-2g-c$. 
Boundary components are topological circles, but our preference is to
call them {\em holes}, as $S$ can be obtained from a $g$-torus by removing
the interiors of $c$ mutually disjoint closed disks. 
A {\em tangle} $T$ in $S$ is a subset of $S$ consisting of a finite set
of unoriented curves, where each intersection is transverse and endowed
with under-over crossing information, such that each 
curve that is open (not closed) has endpoints in $\paS$,
and is otherwise disjoint from $\paS$.  
An $n$-tangle is a tangle with exactly $n$ of its curves open.
A {\em face} of a tangle is a connected component of $S\backslash T$.
For convenience the exposition will focus on framed tangles,
where instead of Reidemeister moves of type 1 one uses those of
type $1'$, also known as ribbon moves.  We will ignore 
other possible refinements such as those where 
link components are oriented or coloured in some way. 
\begin{figure}[!ht]
\begin{center}
\includegraphics[width=7.5cm]{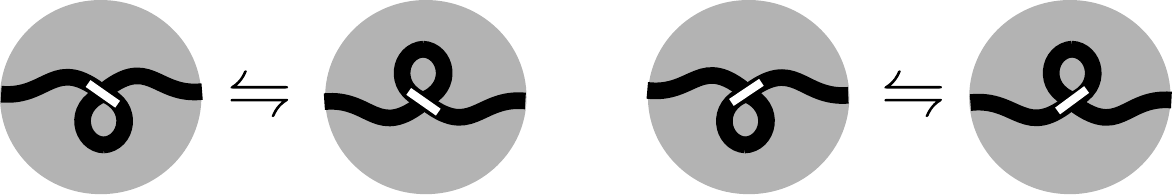}
\caption{\sf Ribbon moves (Reidemeister type $1'$).}
\label{fig:ribbonmoves}
\end{center}
\end{figure}
\begin{definition}
A tangle $T$ in $S$ is {\em fine} if each face $f$ is, topologically, an
open disk or an annulus (or cylinder) whose boundary $\paf$ in $S$
intersects $\paS$ in a connected set, which is a boundary component of $S$
if $f$ is an annulus.
\end{definition}
A {\sl diagram} is an isotopy class of tangles in a surface $S$.
We refer to tangles but implicitly work only up to isotopy.  Tangles that
are not already fine can be adjusted to make them so.  Some choice is
involved but, as explained in \cite{joli2012A}, this can be accounted for.
However, for invariants of the kind we consider, the requirement of
fineness is too stringent, as it is more convenient to work with tangles
satisfying only the following mild condition.
\begin{definition}
A tangle $T$ in $S$ is {\em well-placed} if, for each of its faces $f$, the
boundary $\paf$ intersects $\paS$ in a connected set.
\end{definition}

For a link in $S^2$, the removal of a point in the interior 
of an edge produces a {\em long link}: an infinite version of a 1-tangle
in the plane, having exactly two unbounded faces.
It consists of a long knot and possibly other components that are links. 
Long links in more general surfaces will not be contemplated here.
The invariants below do not exploit these topological distinctions, treating
indifferently links in $S^2$ and links or long links in the plane. 
Our main concern is to avoid tangles that are not well-placed
(adapting the definition when treating long links).
In particular, the only Reidemeister moves admitted are those
that take place between well-placed tangles. 
Whenever tangles arise in other constructions, they should be converted to
well-placed tangles.

Diagrams in surfaces will often be decorated by adding further structure. 
By an order we mean a linear (total) order, unless otherwise stated.
The orientation on $S$ induces a cyclic order on each hole of $S$.
Suppose two mutually disjoint and possibly empty sets of holes
are chosen, and each listed in a fixed order.  These holes are called
respectively the {\em upper} and {\em lower boundaries} of the diagram,
or {\em inputs} and {\em outputs}.
At each hole there is a cyclically ordered set, possibly empty, of
endpoints of curves in the tangle.
We force these orders to become linear by marking a starting point on
each hole, avoiding the tangle.
There is an obvious way to form a category,  where two isotopy classes of
tangles can be composed when outputs of the first are compatible with inputs
of the second.  In order to be able to fuse strands of tangles where a pair
of holes merges, aligned via the marked points, the holes must have
anti-isomorphic linear orders (each isomorphic to the reverse of the other).

One useful variant is the category whose morphisms consist of equivalence
classes of well-placed tangles, with equivalence given via Reidemeister
moves between diagrams.  Such definitions ensure good control over the
Euler characteristics of faces obtained when pasting together tangles
in surfaces.
Yet another category is obtained by pasting classes of long links.

\begin{figure}[!ht]
\begin{center}
\includegraphics[width=12cm]{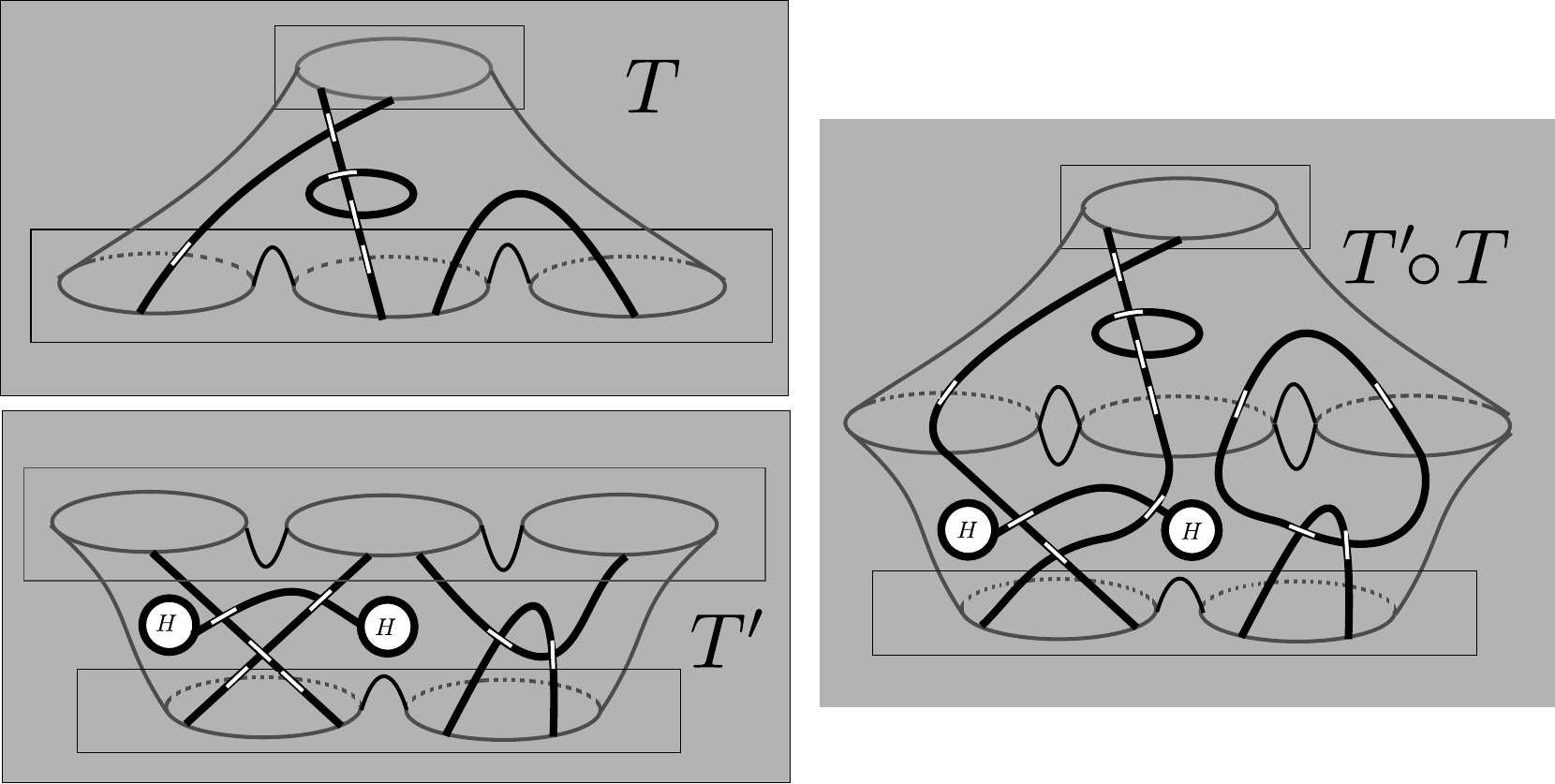} \\
\caption{\sf Composition of tangles (not well-placed) in surfaces with holes.}
\label{fig:generalcomposition}
\end{center}
\end{figure}

More importantly, the pasting procedure can be reversed, say starting with a
tangle in a surface $S$, not assumed to be connected, and a cut that cleaves $S$
into two specified pieces $S_1$ and $S_2$  by removing closed curves disjoint from
each other and from pre-existing holes, having only transverse intersections
with the tangle.  Each curve in a cut is required to be a boundary component
of both $S_1$ and $S_2$.  A tangle $T$ in $S$ gives tangles in $S_1$ and $S_2$
that can be replaced by equivalent well-placed tangles $T_1$ and $T_2$.

One can also cleave surfaces repeatedly, producing (after adjustment)
well-placed tangles $T_i$ in surfaces $S_i$, ordered linearly and compatibly.
When pasted together they produce a tangle in $S$ that is clearly well-placed.
A noteworthy class of examples is given by rational links in the plane, or in
$S^2$.  By removing closed curves that intersect the link in exactly four
points, these can be cut into simple pieces.
One could if desired work within an even more general setting for cutting and
pasting, much like that for the planar algebras of Jones \cite{jones1999planar},
but with surfaces that need not be planar.

\begin{figure}[!ht]
\begin{center}
\includegraphics[width=4.5cm]{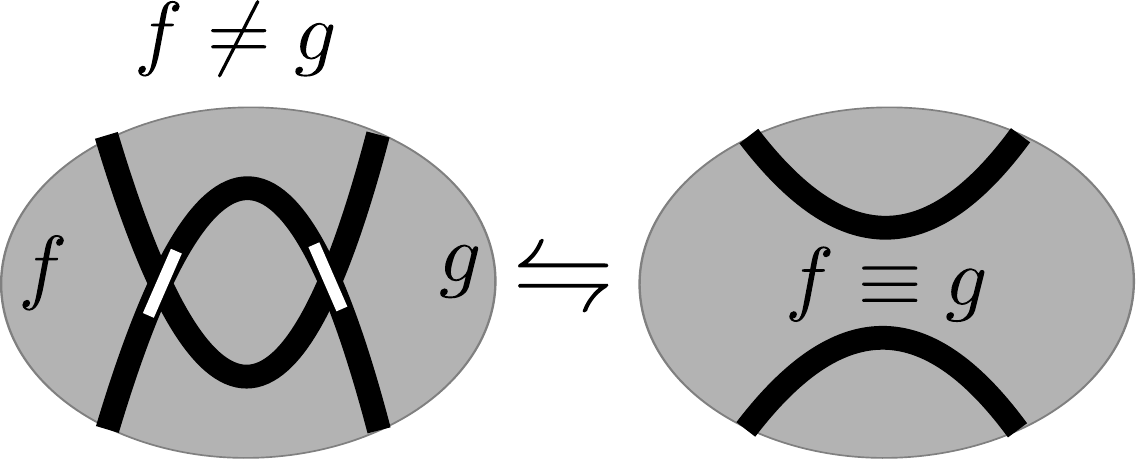}
\caption{\sf Admissible move of type 2 (non-local) for fine tangles.}
\label{fig:unrestricted2}
\end{center}
\end{figure}

\begin{figure}[!ht]
\begin{center}
\includegraphics[width=6.8cm]{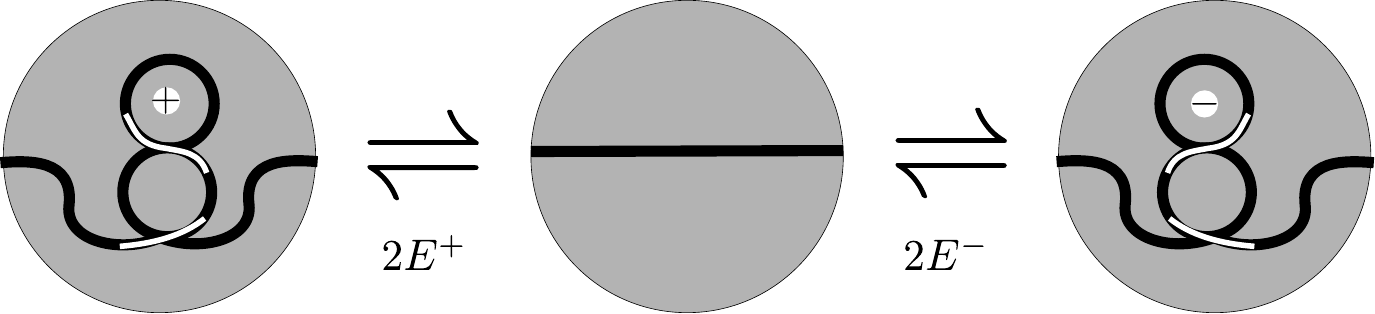} \hspace{6mm}
\includegraphics[width=7.8cm]{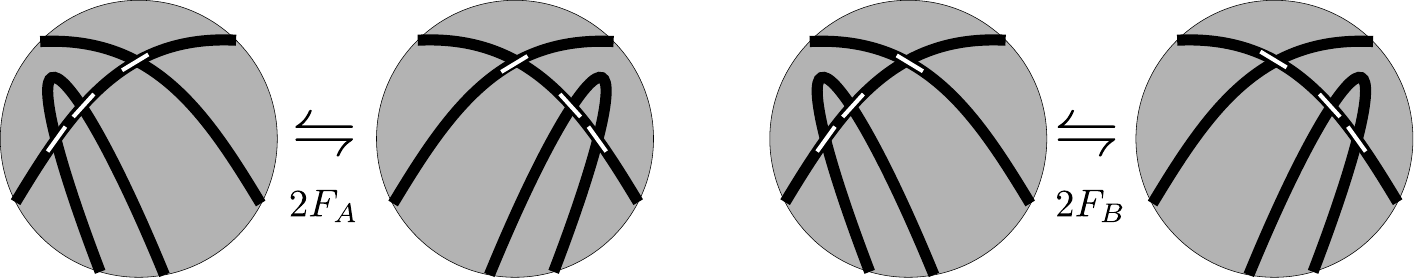} \\
\caption{\sf Local moves that can replace admissible moves of type 2.} 
\label{fig:eigthfingermoves}
\end{center}
\end{figure}

In the context of fine tangles, a Reidemeister move of type 2
is admissible only when the two faces $f$ and $g$ are distinct,
as shown in  Fig.~\ref{fig:unrestricted2},
and at least one of these has boundary disjoint from $\paS$.
Such conditions are useful in practice even though they are non-local:
unlike Reidemeister moves of type 1 or 3, these ones
cannot be verified merely by examining the given parts of the diagrams.  
The four kinds of move shown in Fig.~\ref{fig:eigthfingermoves}, which are 
clearly local, form an adequate substitute for admissible
moves of type 2, as moves between fine tangle diagrams.
This is proved in Theorem~2.2 of \cite{joli2012A}.
However, to justify results for the wider class of well-placed tangles, it
will be necessary to examine invariance under more general moves of type 2.

Finally, to see a useful composition of moves, consider a fine tangle
having only one pair of external connections (a $1$-tangle),
as illustrated in Fig.~\ref{fig:tanglepassingcrossing},
where the letters in the faces can for now be ignored.
It is not hard to see that, via sequences of admissible moves, $T$ can
be moved through a crossing of either type (i.e., regardless of how
the strands cross), provided $T$ lies within a topological disk.
Non-planar tangles do not have this property.

\begin{figure}[!ht]
\begin{center}
\includegraphics[width=5cm]{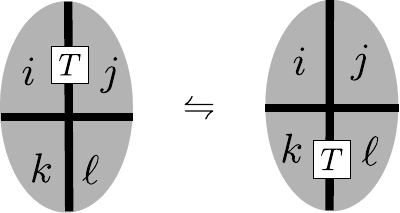}
\caption{\sf A planar 1-tangle can pass through a crossing.}
\label{fig:tanglepassingcrossing}
\end{center}
\end{figure}

\section{Facial state-sum evaluations of tangles}
\label{sec:statesum}

The first ingredient in the present approach to constructing state-sum
invariants is an arbitrary finite set whose elements we call symbols,
or coloured dots. 
A {\em state} of a tangle diagram is a function that assigns a symbol to
each face.  
Something more general is needed, so by a {\em partial state} we mean a
relation that assigns zero or more symbols to each face or, informally,
places coloured dots in some faces.
These will be assigned values in some fixed ring $\cR$ (when the tangle
is a link) or in a certain overring $\Rsq$, described below.
A partial state is inconsistent, and is assigned value zero,
if it has a face with dots of different colours.  Multiple copies
of dots (same colour, in the same face) have no effect on values. 
Any assignment of a values to states can be
extended to an assignment on partial states.  One could just sum over all
states that extend a given partial state, but we find it more useful to
follow the convention of not counting the contribution (defined below)
from faces to which the partial state has already assigned a symbol.

The general approach adopted for assigning a value to each state on
the faces of a tangle diagram is to assign the product of variables that
record local information about the state in various parts of the diagram. 
In what follows we proceed very simply, with nothing more than a variable at
each face and at each crossing, and a notion of forbidding states where
certain symbols lie in adjacent faces. 
Later we will study the relations necessary to ensure invariance of values
under the relevant Reidemeister moves.

For a fine tangle in $S$, each face $f$ is either an {\em internal face}
(its boundary $\paf$ is disjoint from $\paS$) or a {\em boundary face}
($\paf$ meets $\paS$ in a connected set, part or all of the boundary of
a unique hole of $S$).  Internal faces are topological disks, while
boundary faces can be disks or annuli.
Associate with each symbol $i$ an element $x_i$ of $\cR$ that is formally
the square of an element $\sqrt{x_i}$, say in an overring $\Rsq$ of $R$.
The $x_i$ are also called facial or $1j$ variables.
For fine tangles, the {\em contribution} of an internal face having
symbol $i$ is defined to be $x_i$, but for boundary faces it is
$\sqrt{x_i}$ for a disk, 1 for an annulus.
As will be justified by Theorem~\ref{theo:invert}, no generality is lost
by assuming that the $x_i$ are invertible elements of $\cR$.

In theory it suffices to assign a value to each tangle only via
prior conversion into a fine tangle, in situations where this does not
depend on the choices made, but it is more efficient to work directly with
the much wider class of well-placed tangles.  
One can derive formulas appropriate for the more general faces
that appear, but we give these now and will verify them only later,
in Theorem~\ref{theo:newvalue}.
Just as $S$ has an Euler characteristic $\chi_S$, so does each face $f$.
In a state where $f$ has symbol $i$, the contribution at $f$ is now
defined to be $x_i^{\chi_f}$, unless the boundary of $f$ intersects
$\paS$ in part but not all of some component of $\paS$.  In that case,
the previous assignment must be divided by $\sqrt{x_i}$.
This is consistent with the original formulas for $f$ a disk (one boundary,
$\chi_f = 1$) or an annulus $(\chi_f = 0)$.

At each crossing, viewed so that the overpass is the northeast-southwest
strand, and with adjacent faces endowed with symbols $i,j,k,l$,
in anticlockwise order starting from the east, associate a
so-called $4j$-variable $x_{ijkl}$ with that crossing. 
Since tangles are assumed to be unoriented, symmetry forces
identities $x_{klij} = x_{ijkl}$.

At edges it would be natural to define $2j$-variables $x_{ij}$ from
pairs $\{i,j\}$ of symbols.  It appears at first that, at least when
$\paS$ is empty, such variables could be subsumed into the $4j$-variables
at crossings by an argument using half-edges and square roots
of the $x_{ij}$, but that idea fails to deal properly with the trivial
unknot.  Still, by more careful arguments roughly like those used later in
Theorem~\ref{theo:doubleinvariance}, and via Assumption~\ref{ass:ring} below, 
one could reduce to the situation that each $x_{ij}$ has value 0 or 1.  
A more refined approach (among several viable possibilities) would be to
introduce some kind of $6j$-variables that record the six symbols in faces
around the two endpoints of each edge, as in \cite{king2007idealTV}.
There the situation becomes richer but much more difficult to analyse
algebraically.
The position we adopt is to avoid all such variables, giving instead a list
of the pairs $\{i,j\}$ of
symbols such that any state having adjacent faces marked with $i$ and $j$
is assigned value 0.  Such pairs and states are said to be {\em forbidden}
and are implicitly excluded from lists.
All $4j$ (crossing) variables that involve
forbidden pairs can and should be set to zero.

Now consider a tangle $T$ in a surface decorated as above with upper and
lower boundaries.
Let $\upsilon$ and $\lambda$ be partial states that assign symbols
to those faces of $T$ that meet the upper (resp.~lower) boundary.
Recall that no face can meet both boundaries.
Let $\Upsilon(T)$ and $\Lambda(T)$ denote lists, ordered in some
canonical way, of all possibilities for $\upsilon$ and $\lambda$,
respectively, excluding cases that are inconsistent or forbidden.
Thus, for example, the list corresponding to an empty boundary
consists of a single empty substate.
For simplicity, the next formula is given only in the case where $T$ is
fine and $C$ consists of two symbols called white and black,
with corresponding facial variables $x$ and $y$.
The {\em evaluation} of $T$ is the matrix $[T]$ indexed by pairs of substates
in $\Lambda(T) \times \Upsilon(T)$ whose $(\lambda, \upsilon)$-entry
is the sum $$ \sum_\sigma \left\{ x^{w(\sigma)} y^{b(\sigma)}
\sqrt{x}^{w'(\sigma)} \sqrt{y}^{b'(\sigma)} \xi_\sigma \ | \ \sigma
\text{\ is a state of $T$ that extends both\ }\lambda \text{\ and\ } \upsilon 
\ \right\} \in \Rsq,$$
where $w(\sigma)$ and $b(\sigma)$ are, respectively, the number of
internal faces whose symbol in state $\sigma$ is white or black.
Similarly, $w'(\sigma)$ and $b'(\sigma)$ count colours for those
boundary faces that are topological disks, while $\xi_\sigma$
denotes the product of the associated $4j$-variables.
An analogous formula defines evaluations for well-placed tangles and more
general sets of symbols, in a way tailored to make the following theorem hold.

\begin{theorem}
Suppose $T$ and $T'$ are well-placed tangles in decorated surfaces $S$ and $S'$
such that the lower boundary of $S$ (an ordered set of boundary components)
matches the upper boundary of $S'$.
Then the evaluation $[T'\circ T]$ of the composite tangle $T'\circ T$ is the
matrix product $[T']\,[T]$.
\end{theorem}

\begin{proof}
The matrix has been defined so that the given product relates to composition
of morphisms in the category, where the upper boundary of $T'$ in a surface
$S'$ and the lower boundary of $T$ in $S$ are pasted together, as exemplified
in Fig.~\ref{fig:generalcomposition}.  The matrix product corresponds
to grouping terms in state sums for $T'\circ T$ by the restrictions of
states to the set of faces obtained by pasting together a face $f$ in $S$
and a face $f'$ in $S'$, both having the same symbol $i$. 

The only issue is to see that value of a new face obtained is the product
of the values for the corresponding $f$ and $f'$.  The simplest cases
involve pasting two disks at part of a hole, where we verify
$\sqrt{x_i}.\sqrt{x_i} = x_i$, or pasting two annuli along a full
boundary of each, where we verify $1.1 = 1$. 
More generally, faces can have higher genus (which adds under connected sum)
or there can be additional boundary components disjoint from $\paS$.  All
this is reflected in the use of Euler characteristics to assign values to
faces $f$, and the full result follows easily from the two cases treated.
\end{proof}

A noteworthy special case is when a composition $T$ of two or more
tangles lies in a surface $S$ without boundary.  Thus $T$ is a link (it
consists entirely of closed curves), and $[T]$ is a $1 \times 1$ matrix,
although we usually abuse notation by calling its entry $[T]$. 
This lies in $\cR$, as no square roots remain.

\section{Remarks on computing general ring-valued invariants}

In many cases where the rules for state assignment permit few choices,
corresponding to small values of $r$ in state-sum invariants
similar to ones mentioned above,
work has been carried out using algorithmic methods involving Groebner
bases, where the invariants are the normal forms of polynomials
modulo the ideal of relations.
Harder analyses were performed using Singular \cite{greuel2011singular},
while tables of state-sum invariants were generated by programs written
in Mathematica \cite{wolfram2012mathematica}.
Algorithmic aspects will not be discussed here.  However,
before examining the first interesting example that emerged from
our systematic studies of such invariants, we briefly highlight the role
played by well-known fundamental results on rings and fields.

For ring-valued state-sum invariants in general, the most appropriate ring
is a polynomial ring $\cR$ over $\Z$ in the relevant variables, modulo the
smallest possible ideal $I$ of relations that force the desired invariance.
As $\cR$ is Noetherian, $I$ is a finite intersection of primary ideals $I_i$. 
Jointly, the invariants from the $I_i$ have the same power to discriminate
links as does the general invariant, since $\cR/I$ injects into the direct sum
of the $\cR/I_i$.   Whenever $I_i$ is prime, this gives an invariant with
values in a field (the quotient field of $\cR/I_i$).
For an arbitrary primary ideal $I_i$, a similar idea yields an invariant that
can be considered to lie in a ring whose quotient, modulo the nilradical (some
power of which is zero), is a field.  One could just use fields as, 
when nilpotents are factored out, the discriminating power of invariants
(measured on some set of links) rarely weakens. 
Some loss has been observed only for a few invariants that were already weak.
There is, however, a motive for considering nilpotent elements.  One can at times
produce invariants that are not uselessly weak and can be computed relatively
quickly, by adding extra relations to force values of selected variables to be
nilpotent.
In any case, the observations above justify the following standing assumption
on the rings $\cR$ (and also $\Rsq$) in which the invariants take values.

\begin{assumption}
\label{ass:ring}
For some integer $N$, the divisors of zero in $\cR$ are the elements $r$
with $r^N = 0$.
\end{assumption}

Whenever an ideal $I_i$ yields a weak invariant (tested against small tables of
knots and links), it is discarded.  Then, unless $I_i$ is too complicated
to handle well, it is worth finding some way to present values of the invariant
in a better way than as normal forms relative to some Groebner basis.
When the ideal is prime, we prefer to express each value in the associated
quotient field in a canonical way as a quotient of polynomials.
This can in theory be done, as any field of characteristic 0 is a 1-generated
finite extension of a purely transcendent subfield, where free generators
can be chosen from among the original variables.
More generally, primary ideals $I_i$ can also be handled.
Work carried out with Singular produced, in selected cases,
explicit formulas, allowing invariants to be written in convenient forms.
Our interest is in invariants with fairly simple formulas and good
discriminating power.

\section{A simple example: the $u$-invariant}

Recall the general setup for states and partial states, where 
coloured dots (symbols) are placed in faces.
This section is devoted to one simple example of an invariant that uses
only two symbols, called black and white, subject to one rule:
states are forbidden to have adjacent faces receiving a black dot. 
More precisely, if a forbidden state somehow arises it is given value 0,
just as for inconsistent partial states where some face contains dots of
different colours.
Our rule for forbidden states was not imposed arbitrarily, but emerged from
a complete analysis of the case $r=2$ (two symbols), where an ideal of
relations was decomposed as an intersection of primary ideals.

We consider only unoriented tangles, leaving more complex cases for articles
in preparation.  Thus, given a state of a diagram, at each vertex (crossing)
there are at most five possible configurations with a dot in each of the four
(possibly not different) adjacent faces, with associated 
$4j$-variables $\{z,a,b,e,j\}$ as shown in Fig.~\ref{fig:vars}.
There is a symmetry under a $\pi$-rotation, so $d=a$ and $h=b$.
Each letter corresponds to a binary number obtained by drawing the crossing
with the overpass from northeast to southwest, and reading the dots
anticlockwise starting from the east, using 0 for a white dot and 1 for a black dot.
\begin{figure}[ht]
\begin{center}
\includegraphics[width=15cm]{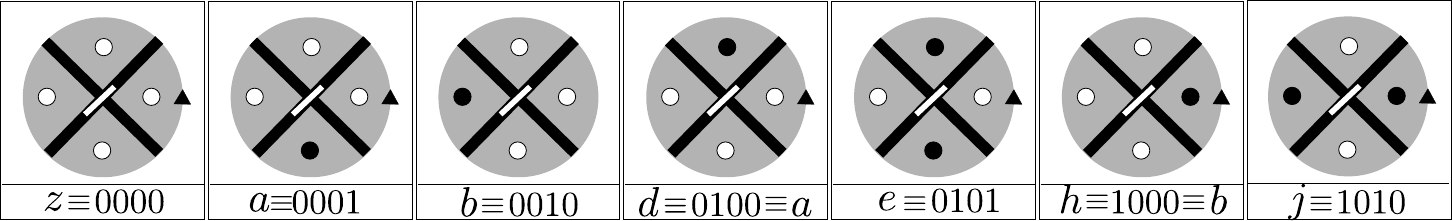} \\
\caption{\sf Scheme to associate variables to crossings}
\label{fig:vars}
\end{center}
\end{figure}

In addition to these five $4j$-variables we use two facial or $1j$-variables:
an $x$ in each white face and a $y$ in each black face, provided the faces are
disks with boundary disjoint from $\paS$.  As mentioned before, and proved
in the general setting of Theorem~\ref{theo:newvalue}, there are related
formulas for other kinds of faces.

When using only the theory of \cite{joli2012A}, the value of a tangle
diagram is calculated only after an adjustment that produces a fine tangle.
To obtain an invariant of framed tangles one need only
verify invariance under Reidemeister moves of types $1'$ (ribbon move), 
2 (restricted) and 3.
The restriction is that the number of faces must change by two under the
move.  In particular, moves that disconnect a diagram are not allowed.

Fig.~\ref{fig:7eqsR2uinv} shows all states, excluding forbidden ones, that
can be assigned to faces in part of a tangle where a Reidemeister move of
type 2 is performed.  In the digon, a dot which is half white and half black
indicates a superposition of two states.
Cases labelled 2 and 8 have rightmost diagrams with inconsistent substates,
so relations are imposed to force the left sides to have value 0. 
As explained below, evaluations here must be performed in a way that
departs from the usual convention for diagrams showing partial states.
Differences produce the following set of six polynomial
relations that must be satisfied in the ring:
$$\cP_2 = \left\{ z^2 x^2 + a b x y - 1,\ a b x^2 - 1, \ z b x + a j y,
\ e j x^2 - 1, \ z a x + b e y, \  a b x y + e j y^2 - 1\right\}.$$
The diagrams form parts of a larger whole where an admissible move takes
place between two fine tangles on a surface.   Three different faces on
the left, one a fully visible digon, merge to form one face on the right.
There may, however, be other coincidences among the four boundary faces
in the leftmost diagrams, so possible contributions of the upper and
lower faces to the evaluation should be ignored for the purposes of
obtaining relations. 
One might expect a common factor of a face variable ($x$ or $y$) in the
relations obtained by forcing pairs of consistent state diagrams to have
equal values, but in fact these do not appear when one of the faces that
fuses is an annulus.  By fineness at least one of the leftmost and
rightmost faces is an internal face, and one sees from this that the
relations already given are sufficient to handle all cases.
We will later see a proof in a general setting that face variables are
invertible, so can be cancelled from relations.

\begin{figure}[!ht]
\begin{center}
\includegraphics[width=14cm]{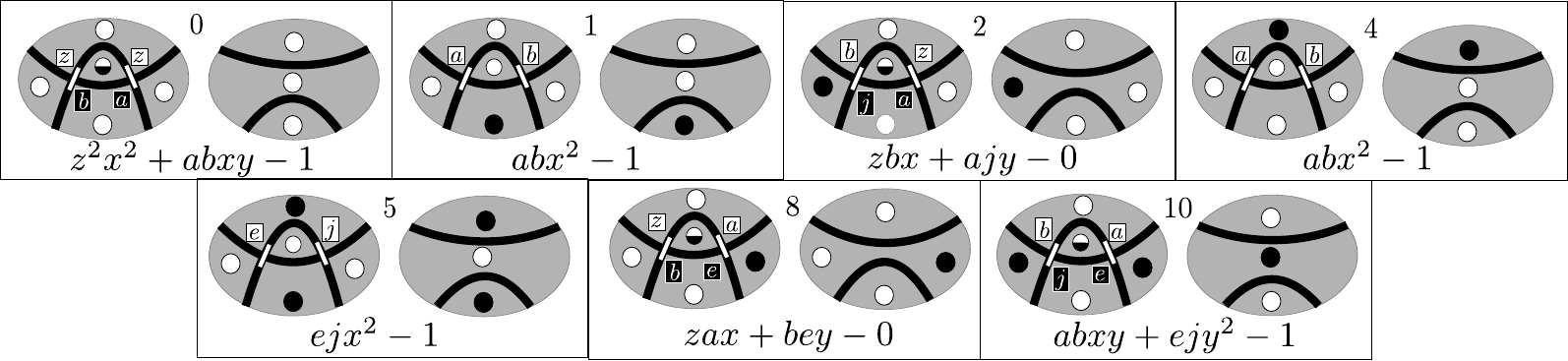} \\
\caption{\sf Relations inducing invariance of tangle evaluation under
admissible non-local moves of type 2.}
\label{fig:7eqsR2uinv}
\end{center}
\end{figure}

\begin{lemma}
Suppose the relations in $\cP_2$ hold for certain elements of a ring $\cR$.
Whenever $T$ and $T'$ are fine tangles differing only by an  
admissible Reidemeister move of type 2, their evaluations in $\Rsq$ are equal.
The same is true for partial states of these tangles, where regions assigned
symbols are not involved in the move.
\end{lemma}

\begin{proof}
In state sums for the leftmost diagram, one groups 
states that agree, except possibly on the digon for the move, with
the corresponding state (if it exists) in the rightmost diagram.
After evaluating, say ignoring contributions from faces with
preassigned symbols, the differences within each group are clearly 
multiples of polynomials in $\cP_2$.
\end{proof}

\begin{figure}[ht]
\begin{center}
\includegraphics[width=14cm]{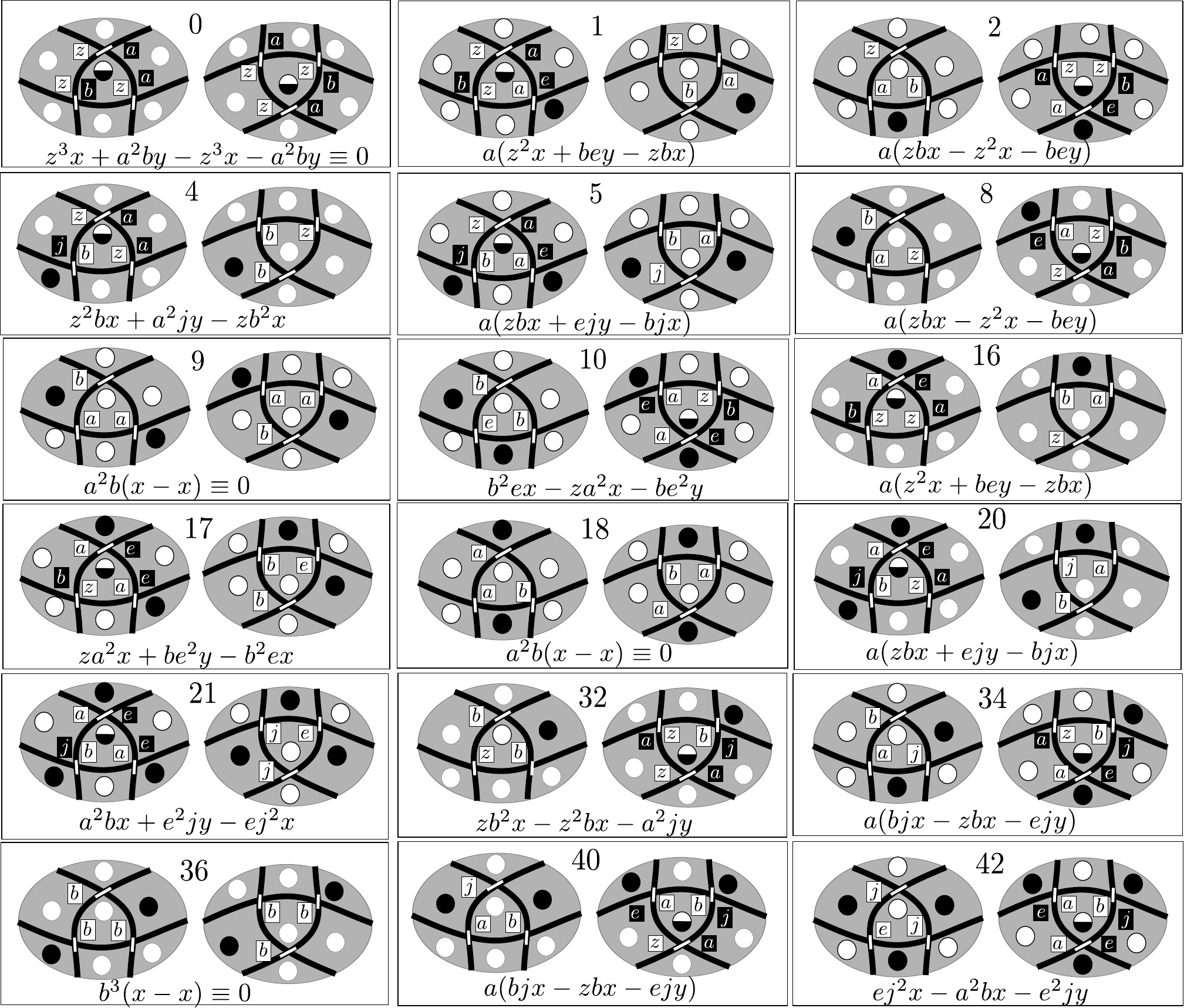} \\
\caption{\sf Polynomial relations for invariance of tangle evaluation
under Reidemeister move 3.}
\label{fig:18eqsR3uinv}
\end{center}
\end{figure}

The Reidemeister move 3 is easy to handle, as non-local features are absent.
Fig.~\ref{fig:18eqsR3uinv} shows 18
pairs of tangles with a complete assignment of states, 
possibly forming parts of larger diagrams.
Differences of values yield polynomial relations
that clearly suffice to guarantee invariance.
Cases 0, 9, 18, 36 give zero, while the others, up to
sign, give a set $\cP_3$ of five distinct relations:
\begin{center}
$\begin{array}{rrrrrr} 
1,2,8,16) & a( z^2 x + b e y - z b x ) & \hspace{10mm}
4) &  z^2 b x + a^2 j y - z b^2 x & \hspace{10mm}
5,20,32,34,40) & a( z b x + e j y - b j x ) \\
10,17)& b^2 e x - z a^2 x -  b e^2 y & \hspace{10mm}
21,42) & a^2 b x + e^2 j y - e j^2 x \\
\end{array}
$
\end{center}

\begin{lemma}
Suppose the relations in $\cP_3$ hold in $\cR$.
Whenever $T$ and $T'$ are fine tangles differing by a Reidemeister move of
type 3, their evaluations in $\Rsq$ are equal.  A similar result holds
for partial states that agree, each leaving unassigned the central
region of the move.
\end{lemma}

The factors of $a$ appearing in some relations of $\cP_3$ can be cancelled
using the second relation in $\cP_2$, which implies that $a$ is invertible.
The ideal generated by $\cP_2 \cup \cP_3$ in the polynomial ring
over $\Z$ with the variables as free generators was analysed using Singular.
It is the intersection of two prime ideals, each giving the same polynomial
invariant up to changes of sign in some variables. 
The value 1 is assigned to $x$, as it will shown near the end of this article
that little information is lost thereby.
By choosing the case where $e$ satisfies $1+e+e^2+e^3+e^4 = 0$, $e$ can be
regarded as a complex primitive fifth root of unity, henceforth called $u$.
From this, it is not
hard to obtain formulas that express the other variables in terms of $u$. 
Thus the invariant (for links rather than tangles in general) can be regarded
as taking values in the subring $\Z[u]$ of $\C$.
It is defined by assigning the following values to the variables.

\begin{definition}
$ 
z \to -u^2-u^3,\ 
a \to -u^3,\ 
b=h\to -u^2,\ 
e\to u,\ 
j\to u^4,\
x\to 1,\ 
y\to u^2+u^3.
$
\label{def:uassignement}
\end{definition}

It is easy to verify by hand that the $u$-assignment,
where $1+u+u^2+u^3+u^4 = 0$, annihilates the 
polynomials in $\cP_2 \cup \cP_3$.   Invariance under the ribbon move
is obvious, so we have now defined an invariant $T \to [T]_u$ of
framed tangles in oriented surfaces, henceforth called the $u$-invariant.
One checks easily that, under a framing change of $+1$, values multiply by $u$.  
Invariance under all Reidemeister moves of type 1 could then be arranged
in the usual way via the self-writhe $sw(T)$ of a tangle diagram $T$
(the sum of the writhes of all curves in $T$).
Thus gives an invariant $u^{-sw(T)}[T]_u$ of unframed tangles that we
prefer not to name, as it suffices to apply the $u$-invariant to
$0$-framed diagrams of tangles.

All values lie in an extension of $\Z[u]$ by a square root of $u^2 + u^3$, but
as our focus is on closed tangles (links $L$) we work only with $\Z[u]$. 
Since $1=-(u+u^2+u^3+u^4)$, each
value can be expressed uniquely in the
form $a u +b u^2 + c u^3 + d u^4$, where $a,b,c,d$ are integers.
We write this as $\lfloor a,b,c,d \rfloor$.  
Its complex conjugate is $\lfloor d,c,b,a \rfloor.$
Mirror images are taken with respect to the surface, altering all crossings.
In both the framed and unframed cases, the corresponding invariants are
complex conjugates of each other, as can be seen from the way
values in $\Z[u]$ were assigned to the variables.  Thus:

\begin{proposition}
If links $L$ and $L^\star$ are mirror images
and $[L]_u = \lfloor a,b,c,d \rfloor$, then
$[L^\star]_u = \lfloor d,c,b,a \rfloor$.
\end{proposition}

\begin{figure}[!ht]
\begin{center}
\includegraphics[width=4cm]{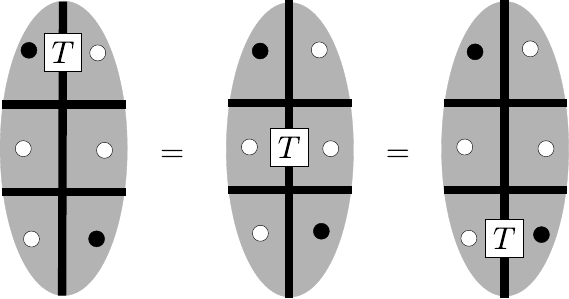}
\caption{\sf A planar 1-tangle passes through crossings with no change in value.}
\label{fig:1tanglepassingcrossingtwice}
\end{center}
\end{figure}
Diagrams are now being used to stand for their evaluations by the original
(framed) $u$-invariant.
Recall our convention in the evaluation of partial states that dots 
already assigned to faces do not contribute facial factors in the state sums.
A result based on moving planar tangles through crossings,
proved in general in Theorem~\ref{theo:doubleinvariance},
and illustrated here as it applies to the $u$-invariant, is:
  
\begin{theorem}
\label{theo:udoubleinvariance}
Let $T$ be a long link in the plane, shown here as an infinite 1-tangle.
Then
\begin{center}
\includegraphics[height=1.3cm]{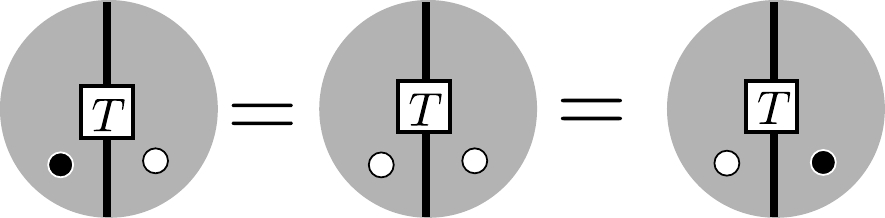}
\raisebox{0.5cm}{ \ and \ $\ [T]_u = \ $}
{\includegraphics[height=1.3cm]{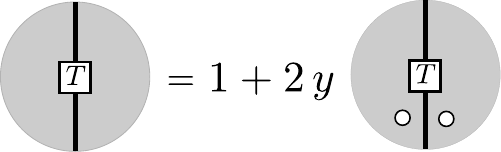}}
\raisebox{0.5cm}{$ \ \in \sqrt{5}\,\Z[u].$}
\end{center}
\end{theorem}

\begin{proof}
The leftmost diagrams have the same value, which also
appears in the state sum for $[T]_u$ with
factors of  $y$, 1, $y$, where $y = u^2+ u^3$.
One checks that $(1+2y)^2 = 5$.
\end{proof}

We write $\{T\}_u$ to denote the leftmost values, so
$\{T\}_u = [T]_u/(1+2y)$.  The new invariant
assigns 1 to the trivial unknot and is clearly multiplicative under
connected sums, so is the preferred form for recording tables of
values for links in the plane, usually adjusted to have self-writhe 0.
The previous result has the following interesting reformulation.

\begin{corollary}
\label{lem:bulletmoves}
For links in the plane or in $S^2$ the value of the $u$-invariant
on a partial state having exactly one dot, coloured black, does not depend on
which face contains the dot.
\end{corollary}

\begin{proof}
It suffices to compare two such values where the relevant faces are
neighbours, for convenience drawn as the infinite faces of a long link,
and treated just above.
\end{proof}

This does not extend to surfaces of positive genus.
The following example of a tangle in a torus has three faces and exactly two
non-zero states with a black dot.  These have different values.
\begin{figure}[!ht]
\begin{center}
\raisebox{-0.2cm}{\includegraphics[height=2.0cm]{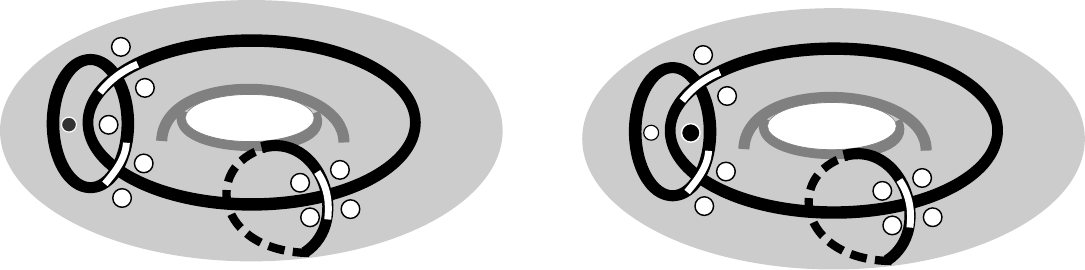}}
\caption{\sf Example in the torus showing that the corollary does not generalize.}
\label{fig:exampleintorus}
\end{center}
\end{figure}

\section{Further properties of the general state-sum invariants}

In what follows, we assume that the variables, regarded as elements of
a ring $\cR$ or $\Rsq$, satisfy relations analogous to those studied above.
Thus they give an invariant of fine tangles and even, by preparation,
an invariant $[T]$ of tangles $T$.

To avoid trivialities, for each symbol $i$ assume there is some $j$ such that
$\{i,j\}$ not forbidden.  If this failed for some $i$, that symbol would be
almost useless, as its only possible contribution could be for
tangles disjoint from some component of the surface they lie in.
We can now verify earlier claims about facial variables $x_i$.
Note that products of these were often cancelled from relations, but
here they should be left in place to justify that practice.

\begin{theorem}
\label{theo:invert}
Under the mild assumption just above,
all facial variables $x_i$ are invertible in $\cR$.
There are formulas $x_i^{-1} =  \sum_k x_kx_{kjij}x_{jijk}$, for suitable $j$.
\end{theorem}

\begin{figure}[!ht]
\begin{center}
\includegraphics[width=1.8cm]{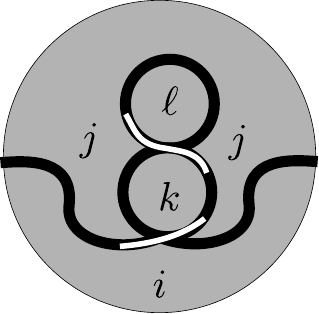} \\
\caption{\sf Eight diagram with state labels.}
\label{fig:eigthmovewithlabel}
\end{center}
\end{figure}

\begin{proof}
Consider state diagrams for the non-local Reidemeister move of type 2,
like those in Fig.~\ref{fig:7eqsR2uinv}, especially the diagram pair marked 2,
but with general states.   Whenever certain boundary regions in the
leftmost diagram have different symbols, giving an inconsistent substate on
the right, the value of the leftmost diagram, given by a state sum, is zero
in $\Rsq$.  Thus, in the following diagram, if $l$ is fixed, as well as $i$ and $j$,
and a sum is taken over $k$, the value obtained is zero unless $l=i$.
Invariance under a 2E (eight) move implies a sum over $k$ and $l$ which
yields a polynomial $p_{ij}$ satisfying
$$\ \sqrt{x_i}\sqrt{x_j} = \sqrt{x_i}\sqrt{x_j}x_ip_{ij}, \text{\ where \ }
p_{ij} = \sum_k x_kx_{kjij}x_{jijk}.$$
The value on the left is not zero, as $\{i,j\}$ is not forbidden.  Then
$1-x_ip_{ij}$ is a divisor of zero, which in turn implies that $\sqrt{x_i}$ is not
nilpotent.  By symmetry, nor is $\sqrt{x_j}$.  By a standing assumption
about the ring, both can be cancelled from the equation.
Thus $x_i^{-1}$ has a polynomial expression.
\end{proof}

Next we return to an alternate method that assigns values directly to states
of well-placed tangles $T$.  Extending the original valuation method, the
contribution of a face $f$ with symbol $i$ is defined to be $x_i^{\chi_f}$,
or this divided by $\sqrt{x_i}$ when $\paf$ contains part but not all of some
boundary component of $S$.

\begin{theorem}
\label{theo:newvalue}
The value assigned directly to a well-placed tangle $T$ is 
the same as that originally obtained only after preparing $T$
to obtain a fine tangle.
\end{theorem}

\begin{proof}
It suffices to show that the new method gives values invariant under the
preparation process, which involves a sequence of Reidemeister moves of
type 2 between well-placed tangles, ending with a fine tangle.  Looking at
any such move made backwards, we focus throughout on the two external faces
that will fuse to the digon and create a new face.  The other faces present
in the move play a minor role, even if they are not distinct from the faces
under consideration, and will be ignored.

First assume that the two faces under study are distinct,
as will hold for example if they contain different symbols.
Globally, the difference in values of diagrams before and after the move is
a multiple of an analogous difference calculated using only parts involved
in the move.  This difference is in turn a multiple, up to powers of facial
variables, of the difference from a move between tame tangles.  Thus, as a
consequence of the relations for tame tangles, this difference has value 0.
Assuming both faces have the same symbol $i$, with contributions $x_i^{f_1}$
and $x_i^{f_2}$, the contribution of the merged face is then
$x_i^{f_1}x_i^{f_2}/x_i,$ and the invariance result is equivalent to the
polynomial formula for $x_i^{-1}$ in Theorem~\ref{theo:invert}.

Now suppose the two faces are, globally, the same face $f$, of genus $g_f$
and with $c_f$ boundary components.   Consider the
two strands of $\paf$ involved in the move.  
There are two cases, illustrated roughly in the diagrams.
In the second, parts of both sides of $f$ are visible.
\begin{itemize}
\item[(1)]  The strands lie in the same component of $\paf$,
which then splits into two, with no change in the genus of the face. 
\item[(2)] The strands lie in different components of $\paf$, which coalesce.
This creates a new handle with a hole, increasing the genus of the face by 1.
\end{itemize}
In both cases the Euler characteristic $2-2g_f-c_f$ of the face decreases by 1,
giving invariance via the polynomial formula for $x_i^{-1}$.
\end{proof}

\begin{figure}[!ht]
\begin{center}
\includegraphics[width=8cm]{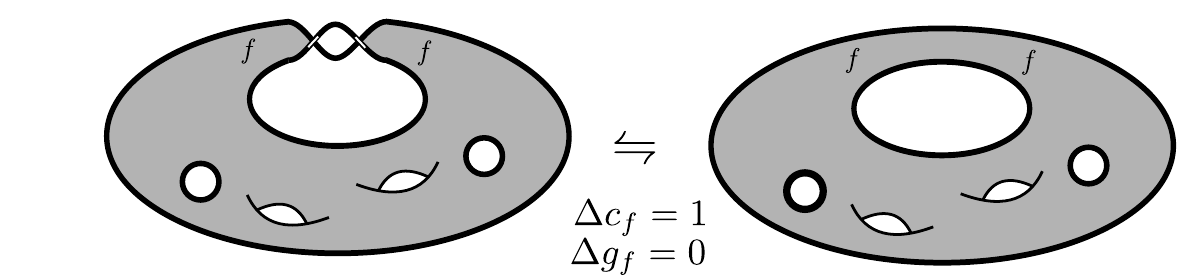} \hspace{10mm}
\includegraphics[width=8cm]{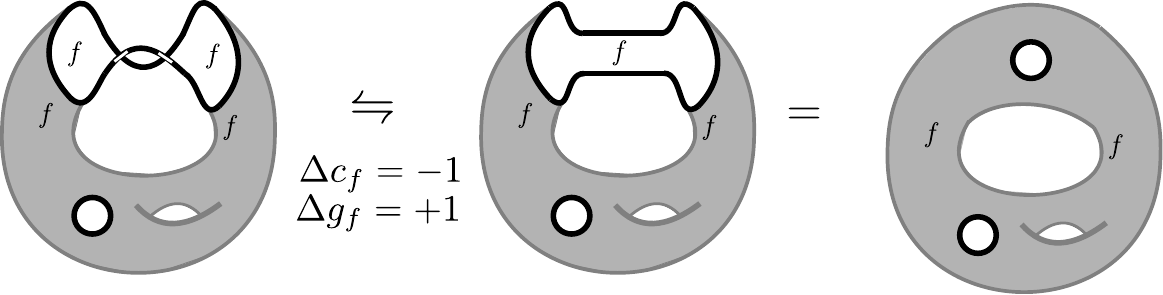} \\
\caption{\sf Cases where two faces that fuse coincide globally.}
\label{fig:samefacecase}
\end{center}
\end{figure}

Some aspects related to moving tangles (especially curls) through crossings
will now be discussed.  The next result generalizes part of
Theorem~\ref{theo:udoubleinvariance} and was used in its proof.

\begin{theorem}
\label{theo:doubleinvariance}
Let $T$ be a long knot or link in the plane, with specified infinite faces
$f$ and $g$, and let $i,j,k,l$ be among the symbols.
Let $[T]_{ij}$ denote the value of the substate which assigns $i$ to $f$
and $j$ to $g$, leaving the other faces unassigned. 
Then $[T]_{ij} = [T]_{kl}$, provided at least one of the $4j$ (crossing)
variables $x_{iklj}$ or $x_{jikl}$ is invertible.
\end{theorem}

\begin{proof}
The diagrams of Fig.~\ref{fig:tanglepassingcrossing} have the same value.
One then cancels a crossing variable.
\end{proof}
 
For arbitrary symbols $i$ and $j$, let $c_{ij}$
denote the factor in $\cR$ calculated from the positive curl in
the next diagram, and let $c'_{ij}$ denote the factor for the positive curl
(not shown) obtained from it by applying a ribbon move.
Explicitly, $c_{ij} = \sum_k x_k.x_{jiki}$, while $c'_{ij}=c_{ji}$ since we
do not orient tangles. 
Invariance under the ribbon move, for diagrams with symbols assigned to the
boundary faces, is thus the assertion that $c_{ij} = c_{ji}$ always holds.
Excluding forbidden pairs $\{i,j\}$ (for which $c_{ij}=0$),
$c_{ij}$ has an inverse in $\cR$, calculable from a negative curl.
This is clear from the Whitney trick.

Theorem~\ref{theo:doubleinvariance} will give certain relations of the form
$c_{ij} = c_{kl}$.
By composing such relations, always avoiding forbidden states, one expects to
see that invariants of regular isotopy produced by these constructions
will often need no further specialization in order to give
invariants of framed tangles.  Certainly this holds
for the $u$-invariant, where the factor for positive curls is $u$.  
In general, however, invariants need not have a single factor measuring
changes of writhe, and examples could well be of interest.

\begin{figure}[!ht]
\begin{center}
\includegraphics[width=4.5cm]{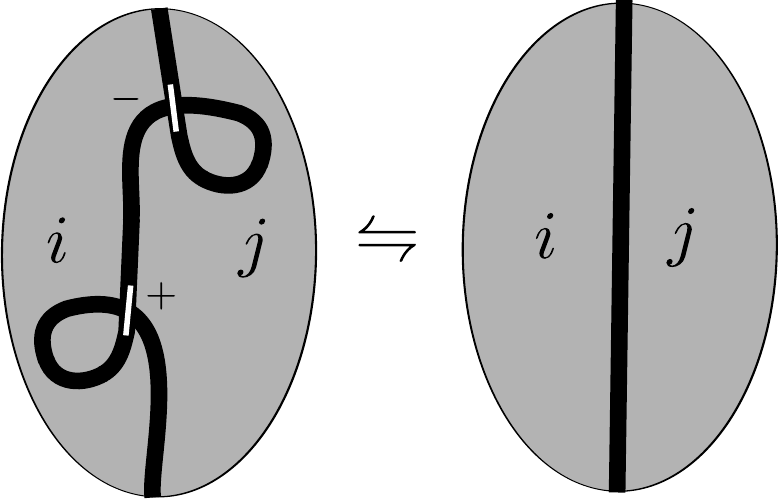}
\caption{\sf The Whitney trick (realizable by admissible Reidemeister moves).}
\label{fig:whitneytrickstate}
\end{center}
\end{figure}

Finally we show, for links in closed surfaces $S$, how little an invariant is affected
by adding a relation $x_0 = 1$, where $0$ denotes some fixed symbol.  Nothing is
lost by taking $\cR$ to be the usual polynomial ring modulo relations.
Instead of adding the above relation, one can introduce new variables, shown
with primes, that satisfy relations $x_i'x_0 = x_i$ (so $x_0'= 1$) and
$x'_{ijkl} = x_{ijkl}x_0$.
All earlier relations for moves between fine tangles can be rewritten in
terms of the new variables and take exactly the same form, as the $x_0$
always cancels completely.  In the case of admissible moves of type 2, as
shown in Fig.~\ref{fig:unrestricted2}, this is because one diagram has
exactly two faces and two vertices more than the other. 

One obtains a new invariant of links on $S$ by giving the primed variables
values in the obvious subring $\cR'$ of $\cR$, and can then identify $\cR$
with the ring of Laurent polynomials $\cR'[x_0,x_0^{-1}]$.
Now observe that whenever the primed invariant assigns value $r'\in \cR'$
to a link $T$ in $S$, the corresponding unprimed invariant has value
$x_0^n.r'$, for some integer $n$.  
 From the definitions of the primed variables, the exponent of $x_0$
is the number of faces of $T$ minus the number of vertices.  But this
is just the Euler characteristic $\chi_S$ of $S$, as every vertex
of $T$ has degree 4.  Nothing more than that can be lost by
assuming $x_0 = 1$.

\section{Appendix A: Values of the $u$-invariant for knots up to
10 crossings}

We record the values of the $u$-invariant, in its normalized form $\{K\}_u$
where the unknot has value 1, for all prime knots $K$ with at most 10
crossings, choosing one in each pair of mirror images and adjusting the
framing to be 0.  Each value can be represented by a vector of four integers
giving the coefficients of $u$, $u^2$, $u^3$ and $u^4$, which reverses
under taking the mirror image.
Here we present this data with the vectors given in condensed form as 4-letter
words, using A,$\dots$,Z for $1,\dots,26$, and a,$\dots$,z for
$-1,\dots,-26$, while 0 is 0.  In addition, write $\alpha$ for -27,
$\beta$ for -30, $\gamma$ for -33.

 From the table, one sees for example that the $u$-invariant distinguishes
$8_{20}$ from its mirror image.  By \cite{tanaka2009} these two knots, as well
as the connected sum of a trefoil with its mirror image, are not distinguished by
knot Floer homology.  On the other hand, the $u$-invariant fails to distinguish
pairs such as $5_1$ and its mirror image, and does not even detect the
unknottedness of $8_{19}$, $9_1$ or $10_{152}$.

{\small
\subsection{Up to 7 crossing knots}
\begin{center}
$
\begin{array}{|lr|lr|lr|lr|lr|lr|lr|}
3_{1} & \text{0aAA} \hspace{0.01mm} &4_{1} & \text{0bb0} \hspace{0.01mm} &5_{1} & \text{ABBA} \hspace{0.01mm} &5_{2} & \text{Aabb} \hspace{0.01mm} &6_{1} & \text{abdb} \hspace{0.01mm} &6_{2} & \text{BBab} \hspace{0.01mm} &6_{3} & \text{beeb} \hspace{0.01mm} \\
7_{1} & \text{AAa0} \hspace{0.01mm} &7_{2} & \text{abdb} \hspace{0.01mm} &7_{3} & \text{BECA} \hspace{0.01mm} &7_{4} & \text{aded} \hspace{0.01mm} &7_{5} & \text{DFEA} \hspace{0.01mm} &7_{6} & \text{DBbc} \hspace{0.01mm} &7_{7} & \text{chfb} \hspace{0.01mm} \\
\end{array}
$
\end{center}

\subsection{8 crossing knots}
\begin{center}
$
\begin{array}{|lr|lr|lr|lr|lr|lr|lr|}
8_{1} & \text{Aabb} \hspace{-0.05mm} &8_{2} & \text{ACEB} \hspace{-0.05mm} &8_{3} & \text{beeb} \hspace{-0.05mm} &8_{4} & \text{DEB0} \hspace{-0.05mm} &8_{5} & \text{AEFD} \hspace{-0.05mm} &8_{6} & \text{eeaB} \hspace{-0.05mm} &8_{7} & \text{CCad} \hspace{-0.05mm} \\
8_{8} & \text{ehea} \hspace{-0.05mm} &8_{9} & \text{dhhd} \hspace{-0.05mm} &8_{10} & \text{EEac} \hspace{-0.05mm} &8_{11} & \text{DAed} \hspace{-0.05mm} &8_{12} & \text{ciic} \hspace{-0.05mm} &8_{13} & \text{dif0} \hspace{-0.05mm} &8_{14} & \text{Daff} \hspace{-0.05mm} \\
8_{15} & \text{FJFa} \hspace{-0.05mm} &8_{16} & \text{eaFF} \hspace{-0.05mm} &8_{17} & \text{elle} \hspace{-0.05mm} &8_{18} & \text{goog} \hspace{-0.05mm} &8_{19} & \text{aaaa} \hspace{-0.05mm} &8_{20} & \text{Aabb} \hspace{-0.05mm} &8_{21} & \text{Abec} \hspace{-0.05mm} \\
\end{array}
$
\end{center}

\subsection{9 crossing knots}
\begin{center}
$
\begin{array}{|lr|lr|lr|lr|lr|lr|lr|}
9_{1} & \text{aaaa} \hspace{-0.1mm} &9_{2} & \text{0bb0} \hspace{-0.1mm} &9_{3} & \text{baBB} \hspace{-0.1mm} &9_{4} & \text{DEB0} \hspace{-0.1mm} &9_{5} & \text{beeb} \hspace{-0.1mm} &9_{6} & \text{daCC} \hspace{-0.1mm} &9_{7} & \text{EFBb} \hspace{-0.1mm} \\
9_{8} & \text{EFBb} \hspace{-0.1mm} &9_{9} & \text{caEE} \hspace{-0.1mm} &9_{10} & \text{GIE0} \hspace{-0.1mm} &9_{11} & \text{BGIC} \hspace{-0.1mm} &9_{12} & \text{Bbhf} \hspace{-0.1mm} &9_{13} & \text{FICb} \hspace{-0.1mm} &9_{14} & \text{Aeif} \hspace{-0.1mm} \\
9_{15} & \text{Bdih} \hspace{-0.1mm} &9_{16} & \text{eaFF} \hspace{-0.1mm} &9_{17} & \text{dBGF} \hspace{-0.1mm} &9_{18} & \text{HJEb} \hspace{-0.1mm} &9_{19} & \text{bilf} \hspace{-0.1mm} &9_{20} & \text{DJLE} \hspace{-0.1mm} &9_{21} & \text{Dbig} \hspace{-0.1mm} \\
9_{22} & \text{cBIH} \hspace{-0.1mm} &9_{23} & \text{GJCd} \hspace{-0.1mm} &9_{24} & \text{fmlc} \hspace{-0.1mm} &9_{25} & \text{ileB} \hspace{-0.1mm} &9_{26} & \text{FCeh} \hspace{-0.1mm} &9_{27} & \text{dloh} \hspace{-0.1mm} &9_{28} & \text{geEH} \hspace{-0.1mm} \\
9_{29} & \text{ILCc} \hspace{-0.1mm} &9_{30} & \text{cmpg} \hspace{-0.1mm} &9_{31} & \text{HCfi} \hspace{-0.1mm} &9_{32} & \text{IFei} \hspace{-0.1mm} &9_{33} & \text{ispe} \hspace{-0.1mm} &9_{34} & \text{eqvj} \hspace{-0.1mm} &9_{35} & \text{dfgb} \hspace{-0.1mm} \\
9_{36} & \text{BIJE} \hspace{-0.1mm} &9_{37} & \text{ajme} \hspace{-0.1mm} &9_{38} & \text{JNFd} \hspace{-0.1mm} &9_{39} & \text{Demj} \hspace{-0.1mm} &9_{40} & \text{lfGL} \hspace{-0.1mm} &9_{41} & \text{ileB} \hspace{-0.1mm} &9_{42} & \text{ABBA} \hspace{-0.1mm} \\
9_{43} & \text{aBCB} \hspace{-0.1mm} &9_{44} & \text{cebA} \hspace{-0.1mm} &9_{45} & \text{ehea} \hspace{-0.1mm} &9_{46} & \text{ca00} \hspace{-0.1mm} &9_{47} & \text{DF0d} \hspace{-0.1mm} &9_{48} & \text{FBcd} \hspace{-0.1mm} &9_{49} & \text{AFIE} \hspace{-0.1mm} \\\end{array}
$
\end{center}

\subsection{10 crossing knots}
\begin{center}
$
\begin{array}{|lr|lr|lr|lr|lr|lr|}
10_{1} & \text{AAa0} \hspace{-0.1mm} &10_{2} & \text{bbaA} \hspace{-0.1mm} &10_{3} & \text{aded} \hspace{-0.1mm} &10_{4} & \text{BECA} \hspace{-0.1mm} &10_{5} & \text{AEFD} \hspace{-0.1mm} &10_{6} & \text{BGIC} \hspace{-0.1mm} \\
10_{7} & \text{Aeif} \hspace{-0.1mm} &10_{8} & \text{0BED} \hspace{-0.1mm} &10_{9} & \text{EEac} \hspace{-0.1mm} &10_{10} & \text{fhbB} \hspace{-0.1mm} &10_{11} & \text{GEbd} \hspace{-0.1mm} &10_{12} & \text{Cahh} \hspace{-0.1mm} \\
10_{13} & \text{0img} \hspace{-0.1mm} &10_{14} & \text{FNMC} \hspace{-0.1mm} &10_{15} & \text{0EIG} \hspace{-0.1mm} &10_{16} & \text{FFae} \hspace{-0.1mm} &10_{17} & \text{ciic} \hspace{-0.1mm} &10_{18} & \text{HFbf} \hspace{-0.1mm} \\
10_{19} & \text{bEJH} \hspace{-0.1mm} &10_{20} & \text{ehea} \hspace{-0.1mm} &10_{21} & \text{CJJC} \hspace{-0.1mm} &10_{22} & \text{ahlh} \hspace{-0.1mm} &10_{23} & \text{FBhi} \hspace{-0.1mm} &10_{24} & \text{imhA} \hspace{-0.1mm} \\
10_{25} & \text{HQPE} \hspace{-0.1mm} &10_{26} & \text{iplb} \hspace{-0.1mm} &10_{27} & \text{Fall} \hspace{-0.1mm} &10_{28} & \text{gibD} \hspace{-0.1mm} &10_{29} & \text{IGbh} \hspace{-0.1mm} &10_{30} & \text{Aipj} \hspace{-0.1mm} \\
10_{31} & \text{fome} \hspace{-0.1mm} &10_{32} & \text{jql0} \hspace{-0.1mm} &10_{33} & \text{fppf} \hspace{-0.1mm} &10_{34} & \text{eeaB} \hspace{-0.1mm} &10_{35} & \text{ahlh} \hspace{-0.1mm} &10_{36} & \text{ahlh} \hspace{-0.1mm} \\
10_{37} & \text{dmmd} \hspace{-0.1mm} &10_{38} & \text{ioia} \hspace{-0.1mm} &10_{39} & \text{FPNE} \hspace{-0.1mm} &10_{40} & \text{HAlk} \hspace{-0.1mm} &10_{41} & \text{JJah} \hspace{-0.1mm} &10_{42} & \text{htvi} \hspace{-0.1mm} \\
10_{43} & \text{hssh} \hspace{-0.1mm} &10_{44} & \text{LJbi} \hspace{-0.1mm} &10_{45} & \text{iwwi} \hspace{-0.1mm} &10_{46} & \text{cebA} \hspace{-0.1mm} &10_{47} & \text{AFIE} \hspace{-0.1mm} &10_{48} & \text{elle} \hspace{-0.1mm} \\
10_{49} & \text{eBJI} \hspace{-0.1mm} &10_{50} & \text{EMME} \hspace{-0.1mm} &10_{51} & \text{HBij} \hspace{-0.1mm} &10_{52} & \text{bFMI} \hspace{-0.1mm} &10_{53} & \text{LMBf} \hspace{-0.1mm} &10_{54} & \text{FJFa} \hspace{-0.1mm} \\
10_{55} & \text{IIaf} \hspace{-0.1mm} &10_{56} & \text{EPQH} \hspace{-0.1mm} &10_{57} & \text{Hamm} \hspace{-0.1mm} &10_{58} & \text{Aipj} \hspace{-0.1mm} &10_{59} & \text{LLag} \hspace{-0.1mm} &10_{60} & \text{mwpb} \hspace{-0.1mm} \\
10_{61} & \text{CE0b} \hspace{-0.1mm} &10_{62} & \text{AHJG} \hspace{-0.1mm} &10_{63} & \text{JIAd} \hspace{-0.1mm} &10_{64} & \text{HHad} \hspace{-0.1mm} &10_{65} & \text{F0ik} \hspace{-0.1mm} &10_{66} & \text{gCNK} \hspace{-0.1mm} \\
10_{67} & \text{akpk} \hspace{-0.1mm} &10_{68} & \text{Cbki} \hspace{-0.1mm} &10_{69} & \text{I0mo} \hspace{-0.1mm} &10_{70} & \text{gbIK} \hspace{-0.1mm} &10_{71} & \text{gttg} \hspace{-0.1mm} &10_{72} & \text{ITQE} \hspace{-0.1mm} \\
10_{73} & \text{IBlm} \hspace{-0.1mm} &10_{74} & \text{Chnj} \hspace{-0.1mm} &10_{75} & \text{nvoc} \hspace{-0.1mm} &10_{76} & \text{DMPH} \hspace{-0.1mm} &10_{77} & \text{Eblj} \hspace{-0.1mm} &10_{78} & \text{FQQF} \hspace{-0.1mm} \\
10_{79} & \text{fppf} \hspace{-0.1mm} &10_{80} & \text{fBML} \hspace{-0.1mm} &10_{81} & \text{iwwi} \hspace{-0.1mm} &10_{82} & \text{IIaf} \hspace{-0.1mm} &10_{83} & \text{mlBI} \hspace{-0.1mm} &10_{84} & \text{npbH} \hspace{-0.1mm} \\
10_{85} & \text{HMIA} \hspace{-0.1mm} &10_{86} & \text{mwpb} \hspace{-0.1mm} &10_{87} & \text{Altm} \hspace{-0.1mm} &10_{88} & \text{j$\alpha\alpha$j} \hspace{-0.1mm} &10_{89} & \text{LCmp} \hspace{-0.1mm} &10_{90} & \text{ktm0} \hspace{-0.1mm} \\
10_{91} & \text{hssh} \hspace{-0.1mm} &10_{92} & \text{KXWH} \hspace{-0.1mm} &10_{93} & \text{0IPK} \hspace{-0.1mm} &10_{94} & \text{JJah} \hspace{-0.1mm} &10_{95} & \text{KBmn} \hspace{-0.1mm} &10_{96} & \text{bqzn} \hspace{-0.1mm} \\
10_{97} & \text{aowo} \hspace{-0.1mm} &10_{98} & \text{HVUJ} \hspace{-0.1mm} &10_{99} & \text{jvvj} \hspace{-0.1mm} &10_{100} & \text{AJPI} \hspace{-0.1mm} &10_{101} & \text{MMai} \hspace{-0.1mm} &10_{102} & \text{alsl} \hspace{-0.1mm} \\
10_{103} & \text{HAlk} \hspace{-0.1mm} &10_{104} & \text{gttg} \hspace{-0.1mm} &10_{105} & \text{OMbj} \hspace{-0.1mm} &10_{106} & \text{LLag} \hspace{-0.1mm} &10_{107} & \text{jzxi} \hspace{-0.1mm} &10_{108} & \text{aINI} \hspace{-0.1mm} \\
10_{109} & \text{iwwi} \hspace{-0.1mm} &10_{110} & \text{MMai} \hspace{-0.1mm} &10_{111} & \text{HTTH} \hspace{-0.1mm} &10_{112} & \text{MMai} \hspace{-0.1mm} &10_{113} & \text{Kbtr} \hspace{-0.1mm} &10_{114} & \text{0pxn} \hspace{-0.1mm} \\
10_{115} & \text{l$\beta\beta$l} \hspace{-0.1mm} &10_{116} & \text{NNak} \hspace{-0.1mm} &10_{117} & \text{qqaK} \hspace{-0.1mm} &10_{118} & \text{kzzk} \hspace{-0.1mm} &10_{119} & \text{o$\alpha$q0} \hspace{-0.1mm} &10_{120} & \text{QRBk} \hspace{-0.1mm} \\
10_{121} & \text{rqBN} \hspace{-0.1mm} &10_{122} & \text{q$\alpha$pA} \hspace{-0.1mm} &10_{123} & \text{n$\gamma\gamma$n} \hspace{-0.1mm} &10_{124} & \text{ABBA} \hspace{-0.1mm} &10_{125} & \text{ABBA} \hspace{-0.1mm} &10_{126} & \text{cebA} \hspace{-0.1mm} \\
10_{127} & \text{EIFA} \hspace{-0.1mm} &10_{128} & \text{Aabb} \hspace{-0.1mm} &10_{129} & \text{ehea} \hspace{-0.1mm} &10_{130} & \text{BBab} \hspace{-0.1mm} &10_{131} & \text{bhie} \hspace{-0.1mm} &10_{132} & \text{ABBA} \hspace{-0.1mm} \\
10_{133} & \text{beeb} \hspace{-0.1mm} &10_{134} & \text{DBbc} \hspace{-0.1mm} &10_{135} & \text{bilf} \hspace{-0.1mm} &10_{136} & \text{ACEB} \hspace{-0.1mm} &10_{137} & \text{Baee} \hspace{-0.1mm} &10_{138} & \text{bCIF} \hspace{-0.1mm} \\
10_{139} & \text{0CC0} \hspace{-0.1mm} &10_{140} & \text{0B0a} \hspace{-0.1mm} &10_{141} & \text{ccaC} \hspace{-0.1mm} &10_{142} & \text{CAba} \hspace{-0.1mm} &10_{143} & \text{dhcA} \hspace{-0.1mm} &10_{144} & \text{hf0E} \hspace{-0.1mm} \\
10_{145} & \text{CBA0} \hspace{-0.1mm} &10_{146} & \text{afkf} \hspace{-0.1mm} &10_{147} & \text{EHC0} \hspace{-0.1mm} &10_{148} & \text{fieA} \hspace{-0.1mm} &10_{149} & \text{HMIA} \hspace{-0.1mm} &10_{150} & \text{EIFA} \hspace{-0.1mm} \\
10_{151} & \text{HFbf} \hspace{-0.1mm} &10_{152} & \text{aaaa} \hspace{-0.1mm} &10_{153} & \text{BBab} \hspace{-0.1mm} &10_{154} & \text{cebA} \hspace{-0.1mm} &10_{155} & \text{Baee} \hspace{-0.1mm} &10_{156} & \text{eaFF} \hspace{-0.1mm} \\
10_{157} & \text{IPJA} \hspace{-0.1mm} &10_{158} & \text{fome} \hspace{-0.1mm} &10_{159} & \text{glfA} \hspace{-0.1mm} &10_{160} & \text{aCFC} \hspace{-0.1mm} &10_{161} & \text{Aabb} \hspace{-0.1mm} &10_{162} & \text{EBef} \hspace{-0.1mm} \\
10_{163} & \text{IGbh} \hspace{-0.1mm} &10_{164} & \text{gojb} \hspace{-0.1mm} &10_{165} & \text{dklg} \hspace{-0.1mm}\\
\end{array}
$
\end{center}
}

\section{Appendix B:  Images of the $u$-invariant of knots up to 10 crossings}
We plot the complex numbers $\{ K \}_u$
for all prime knots $K$ up to 10 crossings evaluated at
$u=e^{\frac{2\,\pi\, i}{5}}$ and $u=e^{4\,\frac{\,\pi\, i}{5}}$, or equivalently
at the conjugates $u=e^{\frac{8\,\pi\, i}{5}}$ and $u=e^{\frac{6\,\pi\, i}{5}}$.
Since all mirror pairs of knots are included, these graphics are symmetric
relative to the real axis.  Moreover we include values for all writhes of
the knots involved, thus giving the image a dihedral 10-fold symmetry.
A curious unexpected fact is the great difference in scale corresponding
to the two values of $u$. 
When non-prime knots are included images become a little denser,
from new products of points in the original graphic.
Other interesting graphics arise by taking
certain products and quotients of the above invariants.

\begin{figure}[!ht]
\begin{center}
\includegraphics[width=100mm]{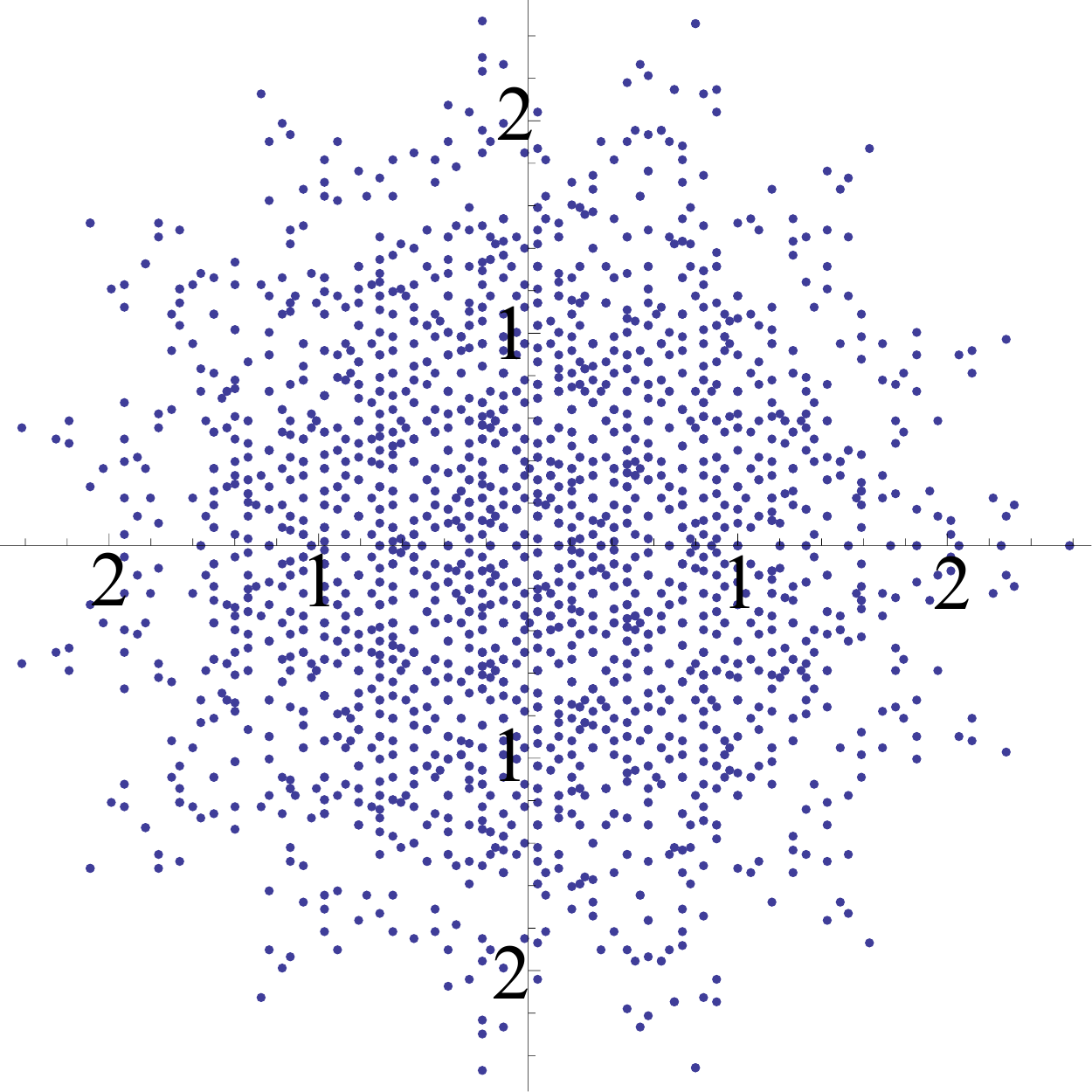} \\
\caption{\sf {\bf The image of the $u$-invariant of knots up to 10 crossings,
computed at $u=e^{\frac{2\,\pi\, i}{5}}.$}}
\label{fig:plot2}
\end{center}
\end{figure}

\begin{figure}[ht]
\begin{center}
\includegraphics[width=100mm]{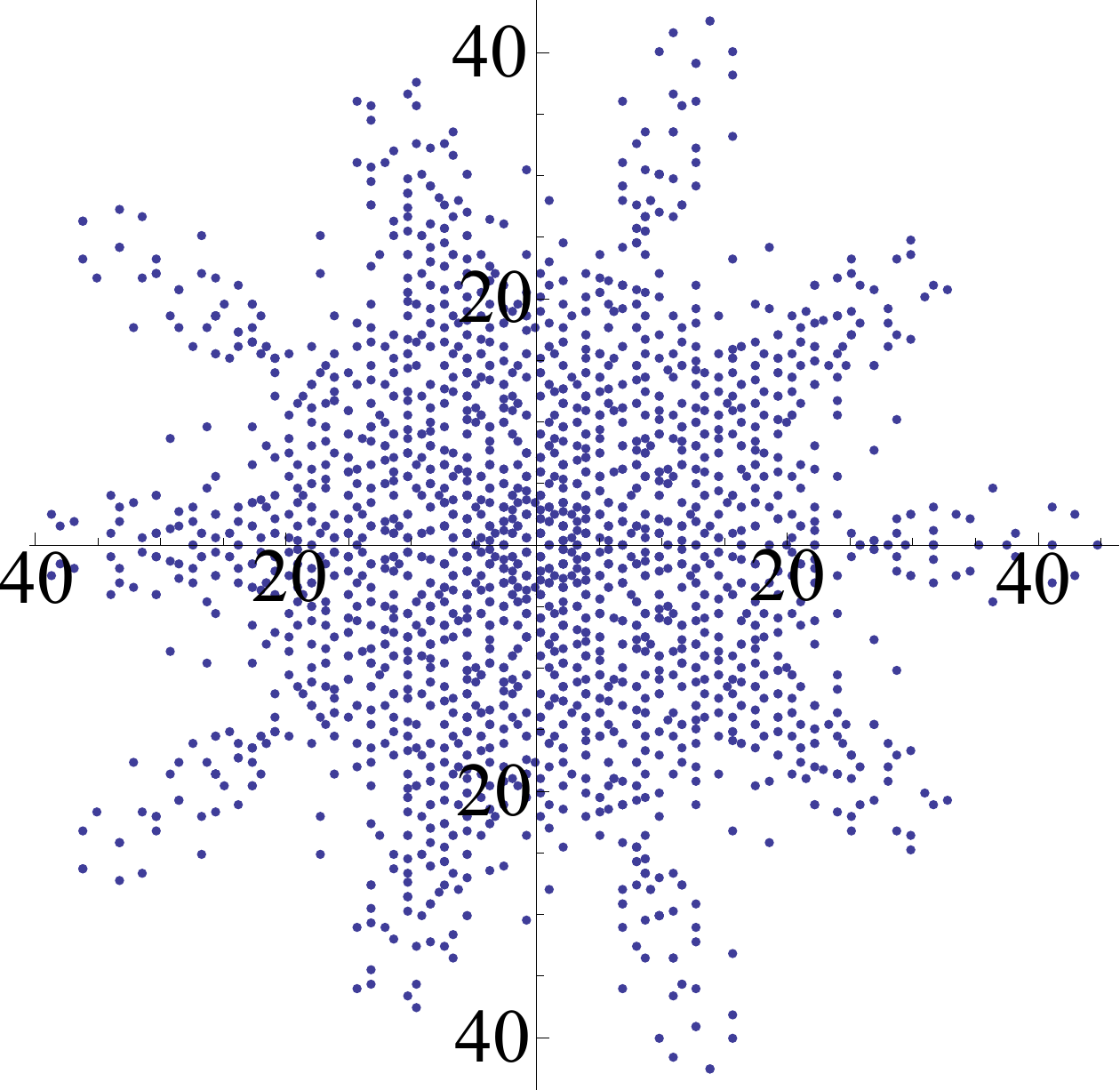} \\
\caption{\sf {\bf The image of the $u$-invariant of knots up to 10 crossings,
computed at $u=e^{\frac{4\,\pi\, i}{5}}.$}}
\label{fig:plot3}
\end{center}
\end{figure}

\bibliographystyle{plain}
\bibliography{bibtexIndex.bib}

\vspace{10mm}
\begin{center}

\begin{tabular}{l}
   Peter M. Johnson\\
   Departamento de Matem\'atica, UFPE\\
   Recife--PE \\
   Brazil\\
   peterj@dmat.ufpe.br
\end{tabular}
\hspace{7mm}
\begin{tabular}{l}
   S\'ostenes Lins\\
   Centro de Inform\'atica, UFPE \\
   Recife--PE \\
   Brazil\\
   sostenes@cin.ufpe.br
\end{tabular}
\end{center}

\end{document}